\documentclass[12pt,a4paper]{article}

\usepackage[utf8]{inputenc}
\usepackage[T1]{fontenc}
\usepackage[english]{babel}
\usepackage{lmodern}

\usepackage{amsfonts}
\usepackage{amsmath}
\usepackage{amssymb}
\usepackage{amsthm}

\usepackage{geometry}
\usepackage{marginnote}

\usepackage{graphicx}
\usepackage{tabularx}
\usepackage{epsfig}
\usepackage{color}
\usepackage{enumerate}
\usepackage{setspace}

\newtheorem{definition}{Definition}[section]
\newtheorem{theorem}[definition]{Theorem}

\newtheorem{proposition}[definition]{Proposition}
\newtheorem{algorithm}[definition]{Algorithm}
\theoremstyle{remark}
\newtheorem{example}[definition]{Example}

\newcommand{\E}{{\mathbb E}}

\newcommand{\R}{{\mathbb R}}

\newcommand{\V}{{\mathbb V}}

\newcommand{\cov}{\mathrm{Cov}}

\begin{document}

\title{Fast Orthogonal transforms for pricing derivatives with quasi-Monte Carlo}

\author{Christian Irrgeher\thanks{Christian Irrgeher, Institut f\"{u}r Finanzmathematik, Universit\"{a}t Linz, Altenbergerstra{\ss}e 69, A-4040 Linz, Austria. e-mail: {\tt christian.irrgeher@jku.at} \quad The author is supported by the Austrian Science Foundation (FWF), Project S9609.}~~and Gunther Leobacher\thanks{Gunther Leobacher, Institut f\"{u}r Finanzmathematik, Universit\"{a}t Linz, Altenbergerstra{\ss}e 69, A-4040 Linz, Austria. e-mail: {\tt gunther.leobacher@jku.at} \quad The author is partially supported by the Austrian Science Foundation (FWF), Project P21196.}}

\date{}

\maketitle

\begin{abstract}
There are a number of situations where, when computing prices of financial derivatives using quasi-Monte Carlo (QMC), it turns out to be beneficial to apply an orthogonal transform to the standard normal input variables. Sometimes those transforms can be computed in time $O(n\log(n))$ for problems depending on $n$ input variables.  Among those are classical methods like the Brownian bridge construction and principal component analysis (PCA) construction for Brownian paths.

Building on preliminary work by Imai \& Tan \cite{imaitan07} as well as Wang \& Sloan \cite{sw10}, where the authors try to find optimal orthogonal transform for given problems, we present how those transforms can be approximated by others that are fast to compute. 

We further present a new regression-based method for finding a Householder reflection which turns out to be very efficient for a wide range of problems. We apply these methods to several very high-dimensional examples from finance.
\end{abstract}


\section{Introduction}\label{sec:intro}

Many simulation problems in finance and other applied fields can be written in the form $\E(f(X))$, where $f$ is a measurable function on $\R^n$ and $X$ is a standard normal vector, that is, $X=(X_1,\ldots,X_n)$ is jointly normal with $\E(X_j)=0$ and $\E(X_jX_k)=\delta_{jk}$. It is a trivial observation of surprisingly big consequences that
\begin{align}\label{eq:fundamental}
\E(f(X))=\E(f(U\!X)) 
\end{align}
for every orthogonal transform $U:\R^n\longrightarrow\R^n$. While this reformulation does not change the simulation problem from the probabilistic point of view, it does sometimes make a big difference when quasi-Monte Carlo simulation is applied to generate the realizations of $X$.

Examples are supplied by the well-known Brownian bridge and PCA constructions of Brownian paths which will be detailed in the following paragraphs. Assume that one wants to know $\E(g(B))$ where $B$ is a Brownian motion with index set $[0,T]$. In most applications this can be reasonably approximated by $\E(\tilde g(B_\frac{T}{n},\ldots,B_\frac{T n}{n}))$, where $\tilde g$ is a function on the set of discrete Brownian paths.

There are three classical methods for sampling from $(B_\frac{T}{n},\ldots,B_\frac{nT}{n})$ given a standard normal vector $X$, namely the forward method, the Brownian bridge construction  and the principal component analysis construction (PCA). All of these constructions may be written in the form $(B_\frac{T}{n},\ldots,B_\frac{nT}{n})=A X$, where $A$ is an $n\times n$ real matrix with 
\begin{align*}
A A^\top=\Sigma:=
\Big(\frac{T}{n}\min(j,k)\Big)_{j,k=1}^n=
\frac{T}{n}\left(
\begin{array}{cccccccc}
1&1&1&\ldots&1\\
1&2&2&\ldots&2\\
1&2&3&\ldots&3\\
\vdots&\vdots&\vdots&\ddots&\vdots\\
1&2&3&\ldots&n
\end{array}
\right)\,.
\end{align*}
For example, the matrix $A$ corresponding to the forward method is 
\begin{align}\label{eq:summation-matrix}
A=S:=
\sqrt{\frac{T}{n}}\left(
\begin{array}{cccccccc}
1&0&\ldots&0\\
1&1&\ldots&0\\
\vdots&\vdots&\ddots&\vdots\\
1&1&\ldots&1
\end{array}
\right)\,,
\end{align}
while PCA corresponds to  $A=V\!D$, where $\Sigma=V\!D^2V^\top$ is the singular value decomposition of\, $\Sigma$. A cor\-res\-ponding decomposition for the Brownian bridge algorithm is given, for example, by Larcher, Leobacher \& Scheicher \cite{lalesch}.

It has been observed by Papageorgiou\cite{papa} that $A A^\top = \Sigma$ if and only if $A=SU$ for some orthogonal matrix $U$, so that every linear construction of $(B_\frac{T}{n},\ldots,B_\frac{nT}{n})$ corresponds to an orthogonal transform of $\R^n$. In that sense the forward method corresponds to the identity, PCA corresponds to $S^{-1}V\!D$ and Brownian bridge corresponds to the inverse Haar transform, see \cite{leo11}.

Thus our original simulation problem can be written, as 
\begin{align*}
\E(\tilde g(B_\frac{T}{n},\ldots,B_\frac{T n}{n}))
=\E(\tilde g(S X))=\E(f(X))
\end{align*}
with $f=\tilde g\circ S$, and we interpret this as using the forward method. Consequently, the same problem using the Brownian bridge takes on the form $\E(f(H^{-1} X))$, where $H$ is the matrix of the inverse Haar transform, and has the form $\E(f(S^{-1}VD X))$, with $S$, $V$, $D$ as above, when PCA is used.

As an application one can generalize the classical constructions of discrete Brownian paths to discrete L\'evy paths. See \cite{leo,lec2008,imaitan09}.
\\

There are some theories as to why an orthogonal transform might have the effect to make the problem more suitable for QMC. Caflisch et al.\ \cite{caflisch97} introduce the concept of effective dimension of a function: consider a function $g:\R^n\longrightarrow \R$ with finite variance w.r.t.\ normal distribution, that is $\E(g(X)^2)<\infty$ where $X=(X_1,\ldots,X_n)$ is a vector of independent standard normal random variables. Then $g$ may be written uniquely as the sum of functions $g_u:\R^n\rightarrow \R$, $u\subseteq\{1,\ldots,n\}$, where $g_u$ depends on the $i$-th coordinate only if $i\in u$ and where $\E(g_u(X))=0$ for all $u\ne\emptyset$ and $\E(g_u(X)g_v(X))=0$ for $u\ne v$, using the so-called ANOVA decomposition of $g$. Furthermore it holds
\begin{align*}
\V(g(X))=\sum_{\emptyset\ne u\subseteq\{1,\ldots,n\}}\V(g_u(X))\,.
\end{align*}
The effective dimension in the truncation sense at level $\alpha\in (0,1)$ is then the smallest integer $k$ such that
\begin{align*}
\V(g(X))(1-\alpha)<\sum_{\emptyset\ne u\subseteq\{1,\ldots,k\}}\V(g_u(X))\,,
\end{align*}
see \cite{caflisch97}. Typically $\alpha$ is chosen as $0.01$. Therefore, a function with effective dimension $k$ is one that, in this sense, almost exclusively depends on the first $k$ variables and which therefore is more suitable for QMC. This is confirmed by empirical evidence. Building on the concept of effective dimension of a function, Owen \cite{owen2012} gives definitions of effective dimensions of function spaces, thus connecting the concepts of effective dimension with that of tractability.  

Now one can turn this around and try to put as much variance as possible to the first few coordinates, by concatenating $g$ with a suitable orthogonal transform. This is what has been done by Imai \& Tan~\cite{imaitan07} and what we will do here, using a different approach. We shall see in Section~\ref{sec:numeric} that empirical evidence also supports the conjectured efficiency of our method.

However, there is also a disadvantage of that approach: the computation of the orthogonal transform has a cost, which is in general of the order $O(n^2)$. For large $n$ this cost is likely to swallow the potential gains from the transform. We therefore concentrate on orthogonal transforms which have cost of the order $O(n\log(n))$.

Examples include discrete sine and cosine transform, Walsh and Haar transform as well as the orthogonal matrix corresponding to the PCA, see \cite{schei,leo11}.
\\

Imai \& Tan \cite{imaitan07} propose an algorithm to find a good orthogonal transform in the sense that it puts as much variance as possible to the first few dimensions. They propose to take the first order Taylor expansion at some point $\widetilde X$, i.e.\
\begin{align*}
g(X)\approx g(\widetilde X)+\sum_{i=1}^{n} \frac{\partial g(X)}{\partial X_i}\vert_{X=\widetilde{X}}(X_i-\widetilde X_i).
\end{align*}
Then the contribution of the $i$-th component of $X$  to $\V(g(X))$ is given by $(\frac{\partial g(X)}{\partial X_i}\vert_{X=\widetilde{X}})^2$. The columns of the orthogonal transform are chosen by solving optimization problems of the form
\begin{align*}
&A^*_{\cdot i}=\max_{A_{\cdot i}\in\R^n}\Bigl(\frac{\partial g(AX)}{\partial X_i}\vert_{X=\widetilde{X}_i}\Bigr)^2\\
\textnormal{with}~&||A_{\cdot i}||=1~\textnormal{and}~A_{\cdot j}^{\top}A_{\cdot i}=0,~j=1,\dots,i-1
\end{align*}
with $\widetilde{X}_i=(\widetilde{X}^{1}_i,\dots,\widetilde{X}^{i}_i,0,\dots,0)$. They suggest to perform this optimization only for the first few columns of the matrix $A$. In this paper we improve on their algorithm in various directions.  In particular  we  find a good orthogonal transform that is fast in that it can be computed even in linear time. \\

The remainder of the paper is organized as follows. Section \ref{sec:householder} reviews basic properties of Householder reflections and shows how they can be used to find fast versions of orthogonal transforms which put most variance on the first $k$ variables. The main part of our article, Section \ref{sec:algo}, describes algorithms for finding fast orthogonal transforms using again Householder reflections. In contrast to the method of Imai \& Tan \cite{imaitan07} we do not rely on differentiability. This makes the algorithm useful for barrier-type options. We further provide some theoretical results which indicate why the method serves to reduce the effective dimension. 

Section~\ref{sec:numeric} gives some numerical examples where the methods described earlier are applied to examples from finance. We will see that the new methods described in Section~\ref{sec:algo} are among the best, both  with regard to speed and accuracy.

We provide an appendix where we compute certain expectations depending on the maximum of a Brownian path. This is useful for some of the numerical examples.

\section{Householder reflections}\label{sec:householder}

We recall the definition and basic properties of Householder reflections from Golub \& van Loan \cite{golub96}.
\begin{definition}
A matrix of the form
\begin{align*}
U =I-2 \frac{v v^\top}{v^\top\!v} \,,
\end{align*}
where $v\in\R^n$, is called a {\em Householder reflection}. The vector $v$ is called the defining {\em Householder vector}.
\end{definition}

In the following proposition, $e_1$ denotes the first canonical basis vector in $\R^n$, $e_1=(1,0,\ldots,0)$.

\begin{proposition}
A Householder reflection have the following properties:
\begin{enumerate}
\item If $x\in\R^n$ is a vector then $Ux$ is the reflection of $x$ in the hyperplane $\mathrm{span}\{v\}^\perp$. In particular, $U$ is orthogonal and symmetric, i.e.  $U^{-1}=U$.
\item Given any vector $a\in\R^n$ we can find $v\in\R^n$ such that for the corresponding Householder reflection $U$ we have $Ua=\|a\|e_1$. The computation of the Householder vector uses $3n$ floating point operations.
\item The computation of the matrix-vector multiplication $Ux$ uses at most $4n$ floating point operations.
\end{enumerate}
\end{proposition}

\begin{proof}
See Chapter 5.1 of Golub \& van Loan \cite{golub96}.
\end{proof}

Our main application of Householder reflections is the following: suppose we know that for a given integration problem $\E(f(X))$ some orthogonal transform $\hat U$ reduces the effective dimension in the truncation sense to $k$, that is, almost all of the variance of $f(\hat U X)$ is captured by $X_1,\ldots,X_k$, $k\ll n$.

Let $\hat U=(\hat u_1,\ldots,\hat u_n)$, that is, $\hat u_j$ is the $j$-th column of $\hat U$. We can find Householder reflections $U_1,\ldots,U_k$ such that $U_1\ldots U_{k} e_\ell=\hat u_\ell$, $\ell=1,\ldots,k$ as follows: 
\begin{itemize}
\item Let $U_1$ be a Householder reflection that maps $e_1$ to $\hat u_1$. $U_1$ also maps $\hat u_1$ to $e_1$. Since the vectors $\hat u_j$ are orthogonal we have $e_1^\top\!(U_1 \hat u_2)=(U_1 \hat u_1)^\top\!(U_1 \hat u_2)=\hat u_1^\top\!\hat u_2=0$.
\item Therefore there exists a Householder reflection $U_2$ operating on the last $n-1$ coordinates which maps $e_2$ to $U_1\hat u_2$. Thus $U_1U_2 e_1=U_1e_1=\hat u_1$, $U_1U_2 e_2=U_1U_1\hat u_2=\hat u_2$. 
\item Suppose Householder reflections $U_1,\ldots,U_j$ have been constructed such that $U_1\ldots U_j e_\ell=\hat u_\ell$, $\ell=1,\ldots,j$.
\item  Then there exists a Householder reflection $U_{j+1}$ operating on the last $n-j$ coordinates which maps $e_{j+1}$ to $U_j\ldots U_1 \hat u_{j+1}$. Then $U_1\ldots U_{j+1} e_\ell=\hat u_\ell$, $\ell=1,\ldots,{j+1}$.
\end{itemize}

Write $U=U_1\ldots U_{k}$. By construction the first $k$ columns of $U$ coincide with those of $\hat U$. Since, by assumption, $X_1,\ldots,X_k$ capture almost all of the variance of $f(\hat U X)$, the same is true for $f(U X)$. But for small $k$ the computational cost for computing $U X$ is of the order $nk$, as compared to general matrix-vector multiplication which occurs a cost of order $O(n^2)$.\\

Imai \& Tan \cite{imaitan07} and Wang \& Sloan \cite{sw10} give examples for which they find good orthogonal transforms $\hat U$ that reduce the effective dimension. However they do not specify how those transforms are applied. We propose to approximate them using the above method. 

However, the main topic of this paper is to present transforms that use only one Householder reflection. This will by detailed in the next section.

\section{Regression algorithm}\label{sec:algo}

Let $f:\R^n\longrightarrow \R$ be a measurable function with $\E(f(X)^2)<\infty$ for a standard normal vector $X$.

We want to approximate $f$ by a linear function:
\begin{align*}
f(x)\approx a^\top x+b 
\end{align*}
where $a\in\R^n$ and $b\in\R$. This can be done in different ways. For example, Imai \& Tan \cite{imaitan07} take the first order Taylor expansion of $f$. 

In contrast, we take a ``linear regression'' approach, i.e.\ we  minimize
\begin{align}\label{eq:lin-crit}
\E\left((f(X)-a^\top X-b)^2\right)\rightarrow \min\,.
\end{align}
First order conditions give
\begin{align*}
a_j=\E(f(X)X_j),\quad j=1,\ldots,n\quad\textnormal{and}\quad b=\E(f(X))\,.
\end{align*}
Therefore, \eqref{eq:lin-crit} minimizes the variance of the difference between $f$ and the linear approximation. So
\begin{align*}
\V(f(X))
&=\E\left((f(X)-b)^2\right)\\
&=\E\left((a^\top X)^2+(f(X)-b-a^\top X)^2\right)\\
&=\|a\|^2+\V\left(f(X)-a^\top X\right)\,.
\end{align*}
That is, $\|a\|^2\!/\V(f(X))$ measures the proportion of variance captured by the linear approximation. Now there exists a unique Householder reflection $U$ that maps $e_1$ to $a/\|a\|$. With this transform we have 
$a^\top U X=\|a\|e_1 ^\top X=\|a\| X_1$ and therefore  
\begin{align*}
\E(f(X))=\E(f(UX))&=\E\left(a^\top U X +\left(f(UX)-a^\top U X\right)\right)\\
&=\E\left(\|a\| X_1 +\left(f(UX)-\|a\|X_1\right)\right)\,.
\end{align*}
Therefore the linear part of the integration problem depends on the parameter $X_1$ alone. Now, if the linear part constitutes a large part of the integration problem then we have succeeded in putting a large fraction of the variance into the first coordinate by composing $f$ with $U$.

\begin{algorithm}
Let $X_1,\ldots,X_n$ be independent standard normal variables.
Let $f$ be a function $f:\R^n\longrightarrow \R$.
\begin{enumerate}
\item $a_j:=\E(X_j f(X))$ for $j=1,\ldots,n$;
\item if $\|a\|=0$ define $U=I$ and go to 4.;
\item else let $U$ be a Householder reflection that maps 
$e_1$ to $a/\|a\|$;
\item Compute $\E(f(UX))$ using QMC.
\end{enumerate}
\end{algorithm}

A drawback of the algorithm is that in general the computation of the expectations in step 1 is no easier than the original problem. In some cases the expectation can be computed explicitly, though usually in that case also the original problem has an explicit solution.

\begin{example}\label{ex:sum_prod}
$f(X) 
=\sum_{k=1}^m w_{k}  \exp\left(\sum_{j=1}^n (c_{k,j}X_j + d_{k,j})\right)$. It is easily verified that, with $\phi$ denoting the standard normal density, $\phi(x)=\exp(-\frac{x^2}{2})/\sqrt{2\pi}$,
\begin{align*}
\int_\R  \exp(cx+d) \phi(x) dx&=\exp\left(c^2/2+d\right)\,,\\
\int_\R x\, \exp(cx+d) \phi(x) dx&=c \exp\left(c^2/2+d\right)\,.
\end{align*}
Therefore it holds that 
\begin{align*}
a_i=\E(f(X)X_i)&=\int_{-\infty}^\infty f(x)x_i\phi(x_1)\ldots\phi(x_n)dx_1\ldots dx_n\\
&=\sum_{k=1}^m c_{k,i}w_{k} \exp\left(\sum_{j=1}^n \left(c_{k,j}^2/2 + d_{k,j}\right)\right)\,.
\end{align*}
Let us find out how much of the variance of $f(UX)$ is captured by $\|a\|X_1$:

We write $\bar w_k:=w_{k} \exp(\sum_{j=1}^n (c_{k,j}^2/2 + d_{k,j}))$. Then 
\begin{align}
\|a\|^2&=\sum_{i=1}^n\left(\sum_{k=1}^m\bar w_k c_{k,i}\right)^2\nonumber\\
&=\sum_{k_1=1}^m\sum_{k_2=1}^m\bar w_{k_1}\bar w_{k_2} \sum_{i=1}^nc_{k_1,i}c_{k_2,i}\nonumber\\
&=\sum_{k_1=1}^m\sum_{k_2=1}^m\bar w_{k_1}\bar w_{k_2} \bar c_{k_1,k_2}\,,\label{eq:asian-a}
\end{align}
where $\bar c_{k_1,k_2}:=\sum_{i=1}^nc_{k_1,i}c_{k_2,i}$.

On the other hand, it is easy to see that $\E(f(X))=\sum_{k=1}^m \bar w_k$ and $\E(f(X)^2)=\sum_{k_1=1}^m\sum_{k_2=1}^m\bar w_{k_1}\bar w_{k_2} e^{\bar c_{k_1,k_2}}$. Therefore we get for the variance of $f(UX)$
\begin{align}
\V(f(UX))&=\V(f(X))=\E(f(X)^2)-\E(f(X))^2\nonumber \\
&=\sum_{k_1=1}^m\sum_{k_2=1}^m\bar w_{k_1}\bar w_{k_2}(e^{\bar c_{k_1,k_2}}-1)\,.\label{eq:asian-v-a}
\end{align}
Let us try some special values that are related to Asian options: 
\begin{align*}
m=n,\; w_k=\frac{1}{n},\;c_{k,j}=\sigma\sqrt{\Delta t}1_{j\le k},\;d_{k,j}=\left(r-\frac{\sigma^2}{2}\right){\Delta t}1_{j\le k}
\end{align*}
with $r,\sigma,T>0$, $\Delta t=\frac{T}{n}$. For this choice we get $\bar w_k=\frac{1}{n}e^{rT k/n}$, and $\bar c_{k_1,k_2}=\sigma^2 T \frac{\min(k_1,k_2)}{n}$. 

For large $n$ the sums in equations \eqref{eq:asian-a} and \eqref{eq:asian-v-a} can be approximated by corresponding integrals such that
\begin{align*}
\|a\|^2&\approx \sigma^2 T \int_0^1\int_0^1 e^{rTx}e^{rTy}\min(x,y)dxdy\\ 
&=\sigma^2  \frac{4 e^{r T}  + 2 e^{2 r T} r T- (3 e^{2 r T}+1)}{2 r^3 T^2}\\
\V(f(X))^2&\approx \int_0^1\int_0^1 e^{rTx}e^{rTy}(e^{\sigma^2 T \min(x,y)}-1)dxdy\\ 
&=\frac{2e^{rT}\left(2r\sigma^2+\sigma^4\right)+2e^{T(2r+\sigma^2)}r^2-\left(e^{2rT}\left(2r^2+3r\sigma^2+\sigma^4\right)+r\sigma^2+\sigma^4\right)}
{r^2 T^2(r+\sigma^2)(2r+\sigma^2)}
\end{align*}
Table \ref{tbl:var-part} shows the fraction $\frac{\V(f(X))-\|a\|^2}{\V(f(X))}$ for a few values of $r$, $\sigma$ and $T=1$.

\begin{table}[ht]
\begin{center}
\begin{tabular}{c{|}cccccccc}
\hline
$r\backslash \sigma^2$&0.01&0.02&0.03&0.04\\\hline
0.1&0.0025& 0.0051& 0.0076& 0.0101\\
0.2&0.0026&   0.0051& 0.0077& 0.0103\\
0.3&0.0026& 0.0052& 0.0078& 0.0104\\
\hline
\end{tabular}
\caption{$\frac{\V(f(X))-\|a\|^2}{\V(f(X))}$ for $T=1$ and different values for 
$r$, $\sigma^2$.}\label{tbl:var-part}
\end{center}
\end{table}

It can be concluded that in this example almost all of the variance of $f(UX)$ is captured by $X_1$.
\qed
\end{example}

In general we cannot expect that $\E(f(X)X_i)$ can be computed explicitly. Of course it is an option to compute $\E(f(X)X_i)$ using (quasi-)Monte Carlo, though it is unlikely that this will lead to small overall computing times. But quite frequently, especially in financial applications,  a problem can be written in the form, $f(X)=g(h(X))$, where $\E(h(X)X_i)$ can be computed and $h$ is some relatively simple function $h:\R\longrightarrow\R$.

\begin{algorithm}\label{alg:main}
Let $X_1,\ldots,X_n$ be independent standard normal variables. Let $f$ be a function $f:\R^n\longrightarrow \R$, which is of the form $f=g\circ h$ where $h:\R^n\longrightarrow \R$ and $g:\R\longrightarrow \R$.
\begin{enumerate}
\item $a_j:=\E(X_j h(X))$ for $j=1,\ldots,n$;
\item if $\|a\|=0$ define $U=I$ and go to 4.;
\item else let $U$ be a Householder reflection that maps 
$e_1$ to $a/\|a\|$;
\item Compute $\E(f(UX))$ using QMC.
\end{enumerate}
\end{algorithm}

Without additional assumptions on the functions $h$ and $g$ there is no guarantee that $U$ gives better convergence. Nevertheless there are practical examples where this algorithm gives excellent results.

\begin{example}\label{ex:asianoption}
Consider an arithmetic average value option written on some underlying $S$,
\[
f(X)=e^{-r T}\max\left(\frac{1}{n}\sum_{k=1}^n S_{\frac{k}{n}T}(X)-K,0\right)\,,
\]
and
\[
S_{\frac{k}{n}T}(X)=S_0\exp\left(\sum_{j=1}^k\sigma\sqrt{\frac{T}{n}}X_j+\left(r-\frac{\sigma^2}{2}\right)\frac{k}{n}T\right)\,.
\]
Here we have $f(X)=g(h(X))$, where $g(s)=e^{-r T}\max(s-K,0)$ and $h$ is like in Example \ref{ex:sum_prod} with $m=n$, $w_{k}=\frac{1}{n}S_0$, $c_{k,j}=\sqrt{\frac{T}{n}}\sigma 1_{j\le k}$, $d_{k,j}=\frac{T}{n}(r-\frac{\sigma^2}{2})1_{j\le k}$\,.\qed\\
\end{example}

%


\vspace{1em}
Write $Y:=\|a\|X_1$, $Z:=h(UX)-\|a\|X_1$. Then $Y,Z$ are uncorrelated, 
\begin{align*}
\E(YZ)=&\E(h(UX)\|a\|e_1^\top X)-\|a\|^2
=\E(h(UX)\|a\|(Ue_1)^\top U X)-\|a\|^2\\
=&\E(h(UX)a^\top U X)-\|a\|^2
=a^\top \E(h(UX)U X)-\|a\|^2\\
=&a^\top \E(h(X)X)-\|a\|^2
=a^\top a-\|a\|^2=0\,.
\end{align*}
Further, $\E(Y)=0$, such that $\E(Y)\E(Z)=0$ as well, and therefore $\cov(Y,Z)=0$.

\begin{theorem}\label{th:decomp}
Let $f,g,h,U,X_1,\ldots, X_n$ be like in Algorithm \ref{alg:main}. Write again $Y:=\|a\|X_1$, $Z:=h(UX)-\|a\|X_1$.

Then $\V(f(UX))=\V(\E(g(Y+Z)|Y))+\V(g(Y+Z)-\E(g(Y+Z)|Y))$.
\end{theorem}

\begin{proof}
We write $\bar Y=\E(g(Y+Z)|Y)$ and $\bar Z=g(Y+Z)-\E(g(Y+Z)|Y)$, so that we have to show $ V(\bar Y+\bar Z)=\V(\bar Y)+\V(\bar Z) $. To that end it is sufficient to prove that $\bar Y$ and $\bar Z$ are uncorrelated:
\begin{align*}
\E(\bar Y\bar Z)
=&\E(\E(g(Y+Z)|Y)g(Y+Z))-\E(\E(g(Y+Z)|Y)\E(g(Y+Z)|Y))\\
=&\E(\E(\E(g(Y+Z)|Y)g(Y+Z)|Y))-\E(\E(g(Y+Z)|Y)^2)\\
=&\E(\E(g(Y+Z)|Y)\E(g(Y+Z)|Y))-\E(\E(g(Y+Z)|Y)^2)=0\,.
\end{align*}
Since $\E(\bar Z)=0$, we have $\E(\bar Y)\E(\bar Z)=0=\E(\bar Y\bar Z)$.
\end{proof}

We consider a special case that will rarely occur in practice but which gives a flavor of the best result possible. Assume that $g$ is Lipschitz continuous with constant $L$. Suppose further that $Y$ and $Z$ are not only uncorrelated, but even independent. 

Denote by $F_Y$, $F_Z$ the cumulative probability distribution functions of $Y$ and $Z$, respectively. Using independence we get
\begin{align*}
\E(g(Y+Z)|Y)=\int_\R g(Y+\zeta) dF_Z(\zeta)\,.
\end{align*}
Noting that $\E(g(Y+Z)-\E(g(Y+Z)|Y))=0$ we thus get
\begin{align*}
\V(g(Y+Z)-\E(g(Y+Z)|Y))
&=\E\left((g(Y+Z)-\E(g(Y+Z)|Y))^2\right)\\
&=\E\left(\left(\int_\R \big(g(Y+Z)-g(Y+\zeta)\big)dF_Z(\zeta)\right)^2\right)\\
&\le\E\left(\int_\R \left(g(Y+Z)-g(Y+\zeta)\right)^2dF_Z(\zeta)\right)\\
&\le\E\left(L^2\int_\R \left(Z-\zeta\right)^2dF_Z(\zeta)\right)\\
&\le L^2\E\left(Z^2-2Z\E(Z)+\E(Z^2)\right)=2 L^2\V(Z)\,,
\end{align*}
where we also have used the Cauchy-Schwarz inequality. Thus with Theorem \ref{th:decomp} we get
\begin{align*}
\V(f(UX))-\V(\E(g(Y+Z)|Y))&\leq 2L^2\V(Z)
\end{align*}
that is,
\begin{align*}
\V(f(UX))-\V(\E(f(UX)|X_1))&\leq 2L^2(\V(h(UX))-\|a\|^2).
\end{align*}
So in this situation, if $X_1$ captures a large fraction of the variance of $h(UX)$, then $X_1$ also captures a large fraction of the variance of $f(UX)$ provided that the Lipschitz constant $L$ is not too big. 

We can also think of a variant of Algorithm \ref{alg:main} for slightly more complicated functions. We have been inspired by Wang \& Sloan \cite{sw10}, where the authors consider functions of the form $f(X)=g(w_1^\top X,\ldots,w_m^\top X)$ and show, that there is an orthogonal transform that makes this function $m$-dimensional. We give a slightly modified version of their argument which guarantees that the orthogonal transform is also fast to compute for small $m$, that is for $m\le\log(n)$.

Let $f(X)=g(w_1^\top X,\ldots,w_m^\top X)$ for $w_1,\ldots,w_m\in\R^n$. We may assume that $w_1$ is not the zero vector. Let $U_1$ be a Householder reflection which maps $(1,0,\ldots,0)$ to $\frac{w_1}{\|w_1\|}$. Then $w_1^\top U_1 X=\|w_1\| (1,0,\ldots,0)^\top X=\|w_1\| X_1$ and therefore 
\begin{align*}
f(U_1 X)=g(\|w_1\| X_1,(U_1w_2)^\top  X,\ldots,(U_1 w_m)^\top X)\,.
\end{align*}
Next we write $(U_1 w_k)^\top X=(U_1 w_k)^\top_1 X_1+(U_1w_k)^\top_{2\ldots n} X_{2\ldots n}$. That is, 
\begin{align*}
f(U_1X)=\bar g(X_1,\bar w_2^\top X_{2\ldots n},\ldots,\bar w_m^\top X_{2\ldots n})
\end{align*}
for some $\bar w_2,\ldots,\bar w_m\in\R^{n-1}$. Assuming that $\bar w_2\ne 0$, let $\bar U_2$ be the Householder reflection of $\R^{n-1}$ that maps $(1,0,\ldots,0)$ to $\bar w_2/\|\bar w_2\|$ and let
\begin{align*}
U_2=\left(
\begin{array}{ccc}
1&0\\
0&\bar U_2
\end{array}
\right)
\,.
\end{align*}
Then $U_2$ is a Householder reflection of $\R^n$ and
\begin{align*}
f(U_1U_2X)=\bar{\bar g}(X_1,X_2,\bar{\bar w}_3^{\;\top} X_{3\ldots n},\ldots,\bar{\bar w}_m^{\;\top} X_{3\ldots n})\,.
\end{align*}
for some $\bar{\bar w}_3,\ldots,\bar{\bar w}_n\in\R^{n-2}$. Proceeding that way one arrives at 
\begin{align*}
f(U_1\cdots U_{\hat m} X)=\hat g(X_1,X_2,\ldots,X_{\hat m})\,
\end{align*}
for some $\hat m\le m$ (We may have $\hat m<m$ if at some stage all remaining $w_k$ are
zero).\\

We propose a similar procedure for an integration problem of the form $f(X)=g(h_1(X),h_2(X),\ldots,h_m(X))$ where $\E(h_j(X)X_k)$ can be computed explicitly (or at least efficiently).

\begin{algorithm}\label{alg:general}
Let $X_1,\ldots,X_n$ be independent standard normal variables. Let $f$ be a function $f:\R^n\longrightarrow \R$, which is of the form $f=g\circ h$ where $h:\R^n\longrightarrow \R^m$ and $g:\R^m\longrightarrow \R$.
\begin{enumerate}
\item Start with $k=1$ and $U=I$;
\item\label{it:goto} $a^{(k)}_j:=\E(X_j h_k(UX))$ for $j=k,\ldots,n$;
\item $a^{(k)}_j:=0$ for $j=1,\ldots,k-1$;
\item if $\|a^{(k)}\|=0$ define $U^{(k)}=I$ and go to \ref{it:final};
\item else let $U^{(k)}$ be a Householder reflection that maps 
$e_k$ to $a^{(k)}/\|a^{(k)}\|$;
\item $U=U U^{(k)}$;
\item $k=k+1$;
\item while $k\leq m$, go back to \ref{it:goto};
\item\label{it:final} Compute $\E(f(UX))$ using QMC.
\end{enumerate}
\end{algorithm}

We will give a numerical example in Section \ref{sec:numeric}.

\section{Numerical tests}\label{sec:numeric}

In this section we will apply our method to examples from mathematical finance.

\subsection*{Asian option}

The first numerical example we give is the evaluation of an Asian call option with discrete arithmetic average in the Black-Scholes model, which has been discussed previously. Since the payoff function $f$ is of the form $g\circ h$ with $g$ and $h$ as in Example \ref{ex:asianoption}, we apply Algorithm \ref{alg:main} to the integration problem $E(f(X))$ where the vector $a$ follows from Example \ref{ex:sum_prod}, i.e. for every $i=1,\ldots,n$
\begin{align*}
a_i=\frac{1}{n}\sum_{k=i}^{n}\sigma\sqrt{\frac{T}{n}}\;e^{rkT/n}.
\end{align*}
For the quasi-Monte Carlo simulation we use a Sobol sequence of dimension $n=250$ with a random shift and we have $S_0=100, K=100, r=0.04, \sigma=0.2$ as well as $T=1$. We compute the standard deviation based on $32$ batches for $N$ sample paths, where the number of sample paths ranges from $2^1$ to $2^{14}$. Note that the standard deviation is different from the RQMC standard deviation defined in L'Ecuyer \& Munger \cite{lecmcqmc}.

In Figure \ref{fig:asian} we compare the regression method with the forward method, the PCA construction and the LT method of Imai and Tan. We see that PCA, LT and the Regression method yield similar results, but all of the three outperform the forward method. Note that the regression method can be applied in $O(n)$. Thus we can achieve the efficiency of the PCA with the regression method with lower computational costs. Moreover, it is interesting that the LT method and regression method yield nearly the same results.

\begin{figure}[ht]
\begin{center}
\begin{minipage}[b]{7.9 cm}
\includegraphics[width=7.9cm,height=5.5cm]{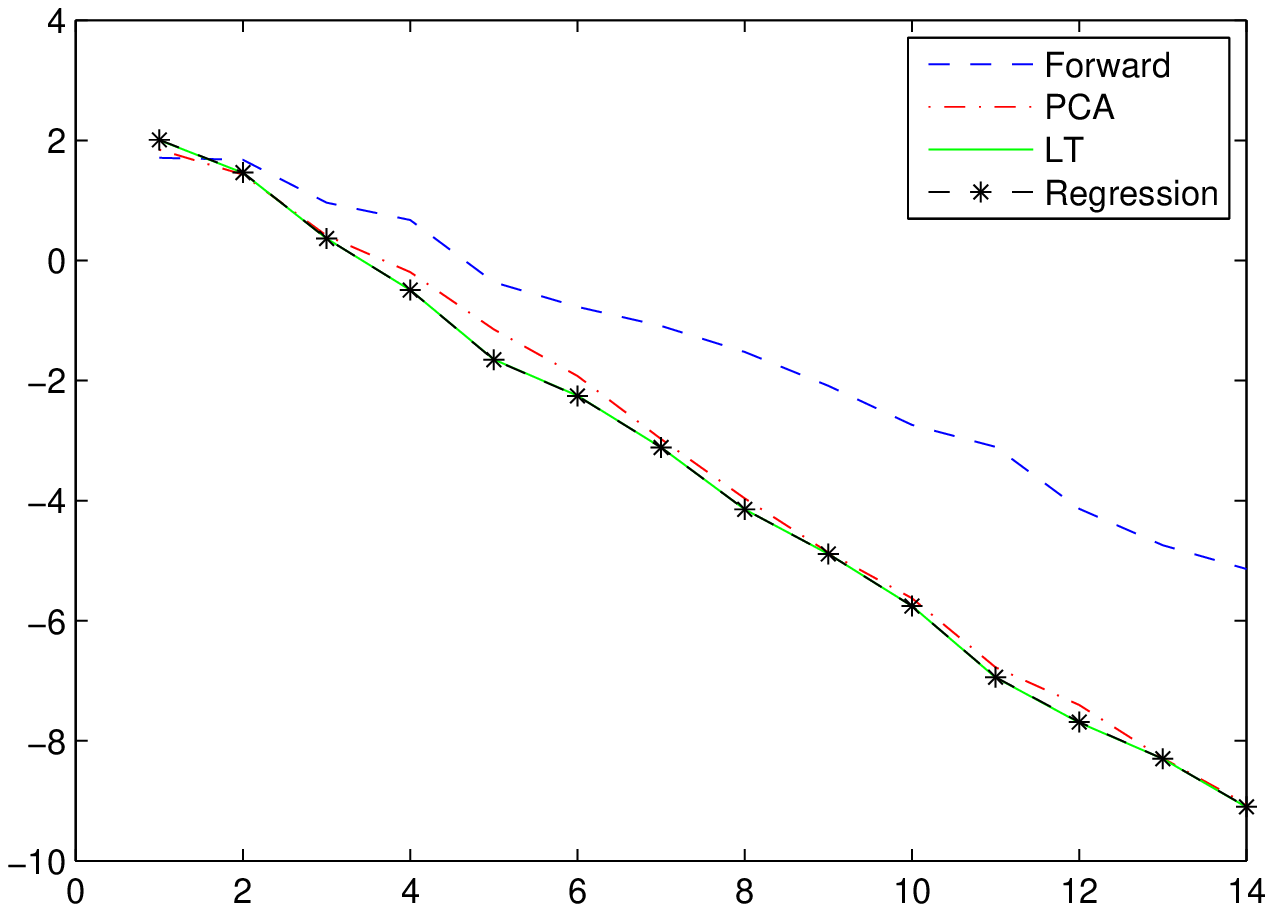}
\end{minipage}
\begin{minipage}[b]{7.9 cm}
\includegraphics[width=7.9cm,height=5.5cm]{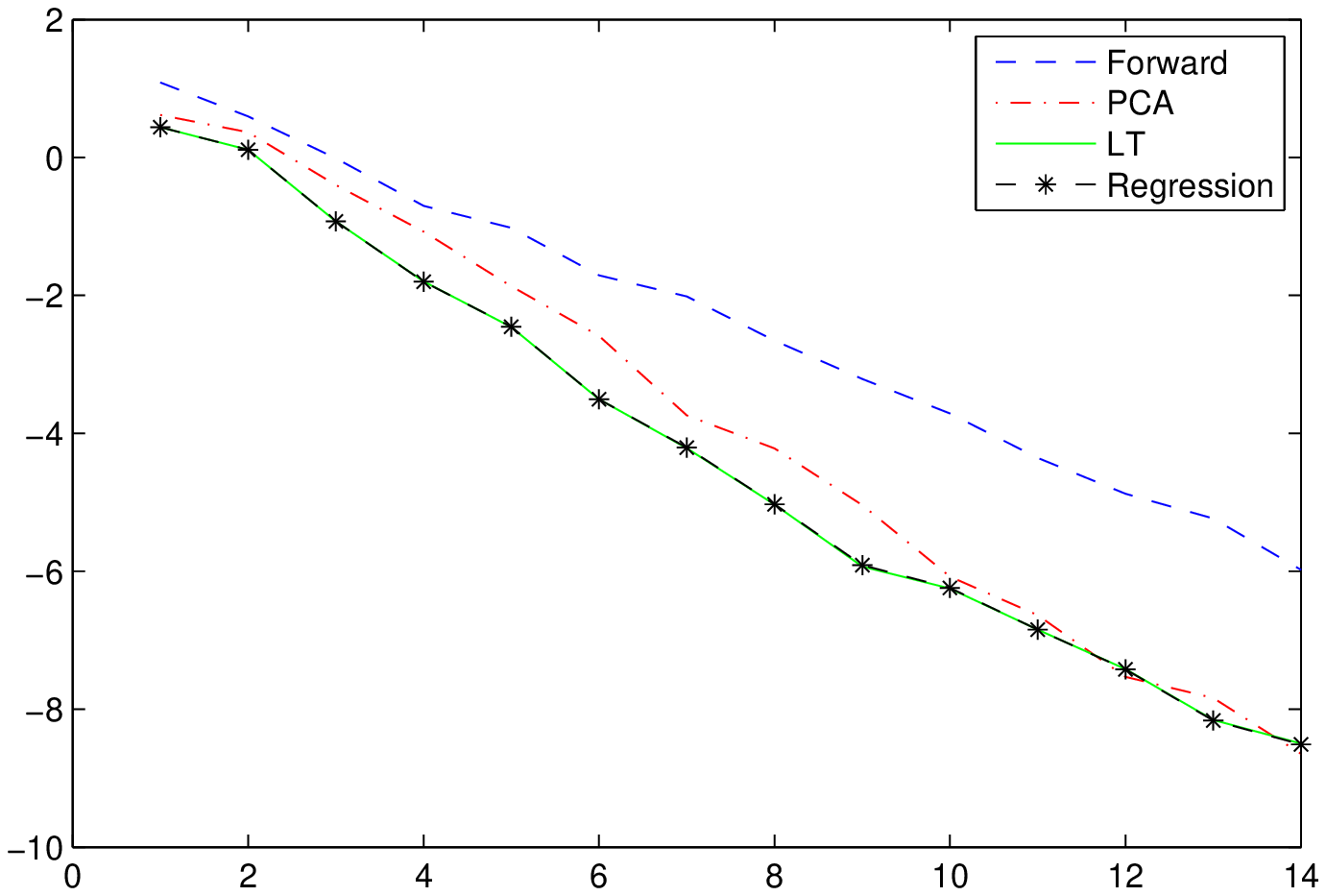}
\end{minipage}
\end{center}
\caption{Asian option (left) and Asian Basket option (right): Standard deviation of $32$ runs on a $\log_2$-scale}
\label{fig:asian}
\end{figure}

The computation time required to price the Asian option using quasi-Monte Carlo integration with $2^{14}$ paths is given in Table \ref{tbl:comptime}. Note that PCA is implemented using the discrete sine transform as discussed in Leobacher \cite{leo11}. The LT method is implemented such that only the first $25$ columns are optimized and then the orthogonal transform is completed using Householder reflections as we suggested in Section \ref{sec:householder}.

\begin{table}[ht]
\begin{center}
\begin{tabular}{c{|}cccc}
\hline
									&\text{Forward}	&\text{PCA}	&\text{LT}	&Regression\\\hline
\text{time (sec)}	&0.08						&0.64				&1.94		 		&0.15 \\
\hline
\end{tabular}
\caption{Computation times for pricing the Asian option}\label{tbl:comptime}
\end{center}
\end{table}

Furthermore it should be mentioned that the regression method as well as the LT method produce an overhead caused by determining the orthogonal transform. Nevertheless the overhead time is rather small and is negligible for a large sample size. 

The computation times of the subsequent numerical examples are similar to the result regarding the Asian option.

\subsection*{Asian basket option}

We consider an Asian basket call option with arithmetic average and a basket consisting of $m$ assets, an example taken from Imai \& Tan \cite{imaitan07}. The $i$-th asset $S^{(i)}$ of the basket ($i=1,\dots,m$) is given by 
\begin{align*}
S_{k\frac{T}{n}}^{(i)}=S_{0}^{(i)}\exp\left(\left(r-\frac{\sigma_{i}^2}{2}\right)k\frac{T}{n}+\sigma_{i}B^{(i)}_{k\frac{T}{n}}\right)
\end{align*}
where $S_{0}^{(i)}$ is the current price of the $i$-th asset, $r$ is the risk-free interest rate, $\sigma_{i}$ is the volatility of the $i$-th asset and $B=(B^{(1)},\dots,B^{(m)})$ is an $m$-dimensional Brownian motion. The correlation between $B^{(j)}$ and $B^{(k)}$ is denoted by $\rho_{jk}$. The payoff function of the Asian basket option is given by
\begin{align*}
f(X)=\max\left(\frac{1}{nm}\sum_{i=1}^{m}\sum_{k=1}^{n}S^{(i)}_{k\frac{T}{n}}(X)-K,0\right),
\end{align*}
where
\begin{align*}
S^{(i)}_{k\frac{T}{n}}(X)=S_{0}^{(i)}\exp\left(\sum_{j=1}^{nm}C_{(k-1)m+i,j}X_j+(r-\frac{\sigma_i^2}{2})k\frac{T}{n}\right),
\end{align*}
and where $C$ is an $mn\times mn-$matrix with $CC^{\top}=\tilde\Sigma:=R\otimes\Sigma$ and $R$ is an $m\times m-$matrix with $R_{ii}=\sqrt{T/n}\,\sigma_{i}^2$ for all $i$ and $R_{ij}=\sqrt{T/n}\,\rho_{ij}\sigma_i\sigma_j$ for $i\neq j$. Note that the discussion of the previous sections also holds for a discrete Brownian path with covariance matrix $\tilde\Sigma$. Since the problem is of the form $f(X)=g(h(X))$, Algorithm \ref{alg:main} can be applied. Since the function $h$ is of the form considered in Example \ref{ex:sum_prod}, we can compute the corresponding vector $a$ analytically. Furthermore, notice that the PCA construction can be computed in this example efficiently by using the orthogonal transform $V_1D_1\otimes V_2D_2$ where $V_1D_1^2V_1^{\top}=R$ and $V_2D_2^2V_2^{\top}=\Sigma$.

The parameters are $T=1, r=0.04, K=100$ and $\rho_{jk}=0.05$ for $j\neq k$. Moreover, the volatility of the $10$ assets is equally spaced from $0.1$ to $0.3$ and we assume that $S_{0}^{(i)}=100$ for all $i=1,\dots,m$. Since we simulate every asset at $250$ time points, we take a Sobol sequence in dimension $n=2500$ with a random shift. In Figure \ref{fig:asian} we can observe the standard deviation based on $32$ batches of the forward method, the PCA construction, the LT method and the regression method for $N$ sample paths with $N$ up to $2^{14}$.

\subsection*{Digital barrier option}

A digital (up-and-in) barrier option is a derivative which pays $1$ if the underlying asset breaks through a barrier $u$ on the time interval $[0,T]$ and pays $0$ otherwise. We intend to price the option in a discrete Black-Scholes model, where the path of the stock is given by $S=(S_1,\dots, S_n)$ with
\begin{align}\label{eq:discretestock}
S_k(X)=S_0\exp\left(\left(r-\frac{\sigma^2}{2}\right)k\frac{T}{n}+\sigma B_{k\frac{T}{n}}\right)
\end{align}
with current stock price $S_0$, interest rate $r$, volatility $\sigma$, Brownian path $B=(B_{k\frac{T}{n}})_{k=1}^{n}$ where $B_{k\frac{T}{n}}=\sqrt{\frac{T}{n}}\sum_{j=1}^{k}X_{j}$ and standard normal vector $X=(X_1,\dots,X_n)$. Hence, the payoff function $h$ of the digital barrier option is 
\begin{align*}
h(X)=1_{\max_{k=1,\dots,n} S_{k}(X)\geq u}
\end{align*}
which leads us to an integration problem of the form $\E(\exp(-rT)h(B))$. We can use Algorithm \ref{alg:main} for solving this problem and therefore, we have to compute $a_{i}=\E(h(X)X_i)$ for $i=1,\dots,n$. In the appendix we show how to calculate this expectation for a function depending on the maximum of a Brownian motion with drift $\nu$. We can adjust our problem by 
\begin{align*}
\max_{k=1,\dots,n} S_{k}\geq u&\Longleftrightarrow\max_{k=1,\dots,n} S_0\exp\left(\left(r-\frac{\sigma^2}{2}\right)k\frac{T}{n}+\sigma B_{k\frac{T}{n}}\right)\geq u\\
&\Longleftrightarrow\max_{k=1,\dots,n} \frac{(r-\frac{\sigma^2}{2})}{\sigma}k\frac{T}{n}+B_{k\frac{T}{n}}\geq \frac{\log(u/S_0)}{\sigma}\\
&\Longleftrightarrow\max_{k=1,\dots,n} B_{k\frac{T}{n}}^{\nu}\geq \tilde{u}
\end{align*}
with $B^\nu_t=\nu t+B_{t}$, $\nu=\frac{(r-\frac{\sigma^2}{2})}{\sigma}$ and $\tilde{u}=\frac{\log(u/S_0)}{\sigma}$. With $(\ref{eq:barrier})$ we get that the vector $a$ in Algorithm \ref{alg:main} can be approximated by
\begin{align*}
a\approx S^{-1}\beta-\nu\sqrt{T/n}\,\gamma\,\mathbf{1}
\end{align*}
where $S$ is given by (\ref{eq:summation-matrix}), $\beta=(\beta_1,\dots,\beta_n)^{\top}$ with $\beta_{i}=\E(1_{\max_{0\leq s\leq T}B_{s}^{\nu}\geq u}B^\nu_{i\frac{T}{n}})$, $\gamma=\E(1_{\max_{0\leq s\leq T}B_{s}^{\nu}\geq u})$ and $\mathbf{1}=(1,\dots,1)^{\top}$. The computation of $\beta_i$ with $i=1,\dots,n$ can be reduced to a $1$-dimensional integration problem using (\ref{eq:refl3}) with $f=\mathrm{id}_{\R}$ and $t=i\frac{T}{n}$ and formula (\ref{eq:refl1}) with $f\equiv1$ simplifies $\gamma$. Consequently, we end up with $1$-dimensional integrals which can be evaluated efficiently with an adaptive quadrature rule.

For the numerical test we use a Sobol sequence of dimension $n=2000$ with a random shift and the parameter set is chosen as $S_0=100, u=110, r=0.04, \sigma=0.2$ and $T=1$. The number of sample paths $N$ ranges from $2^1$ to $2^{14}$ and we compute the standard deviation for those $N$ based on $32$ batches. Since it is not clear how to apply the LT method of Imai and Tan to barrier options, we compare the regression method with the forward method and the PCA construction only. In Figure \ref{fig:barrier} we can observe that the difference between the forward method and the PCA is smaller than in the previous examples. Furthermore, we see that the regression method is slightly behind the PCA, but this seems to be the best we can achieve by linear approximation.

\begin{figure}[ht]
\begin{center}
\begin{minipage}[b]{7.9 cm}
\includegraphics[width=7.9cm,height=5.5cm]{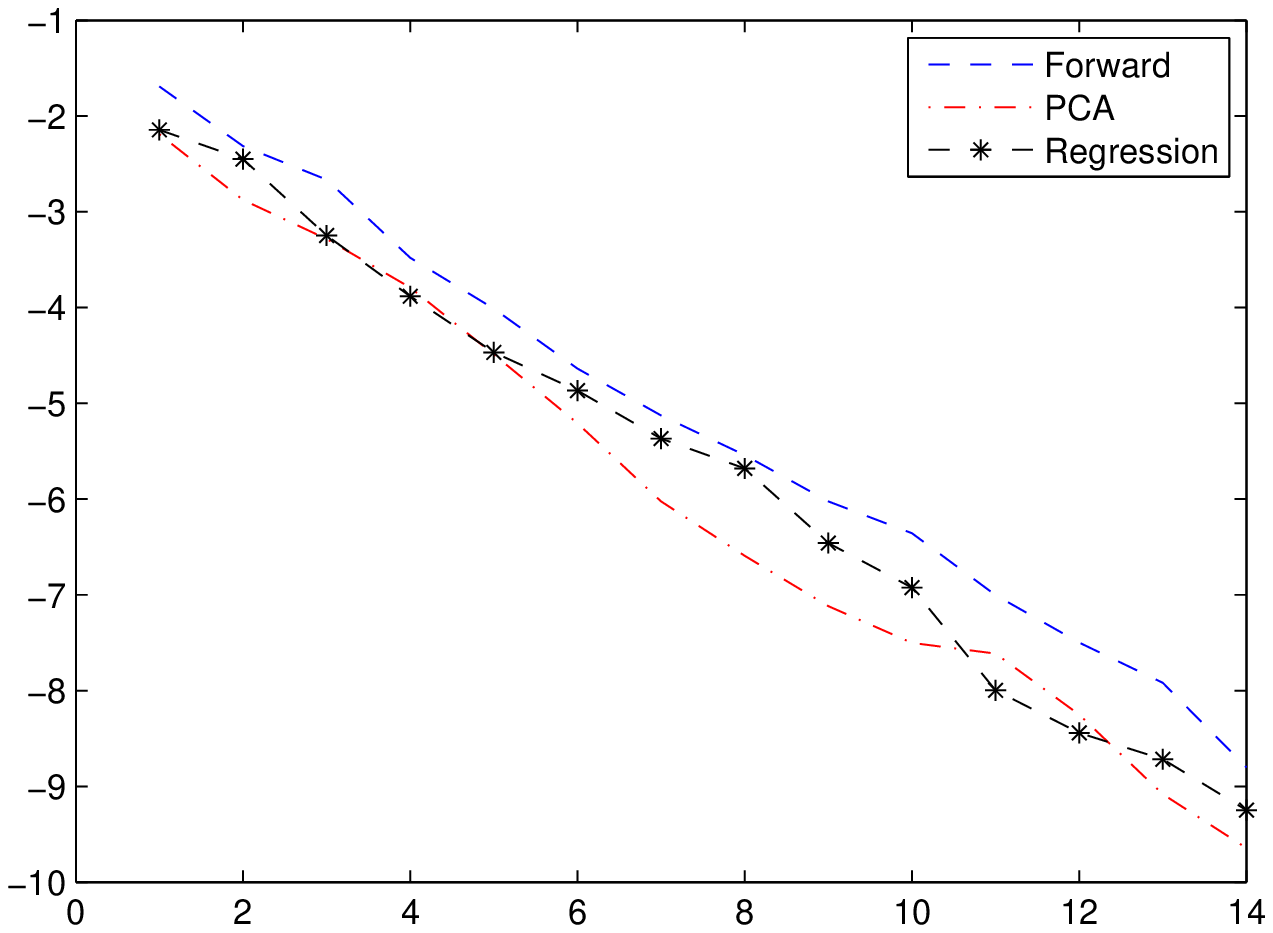}
\end{minipage}
\begin{minipage}[b]{7.9 cm}
\includegraphics[width=7.9cm,height=5.5cm]{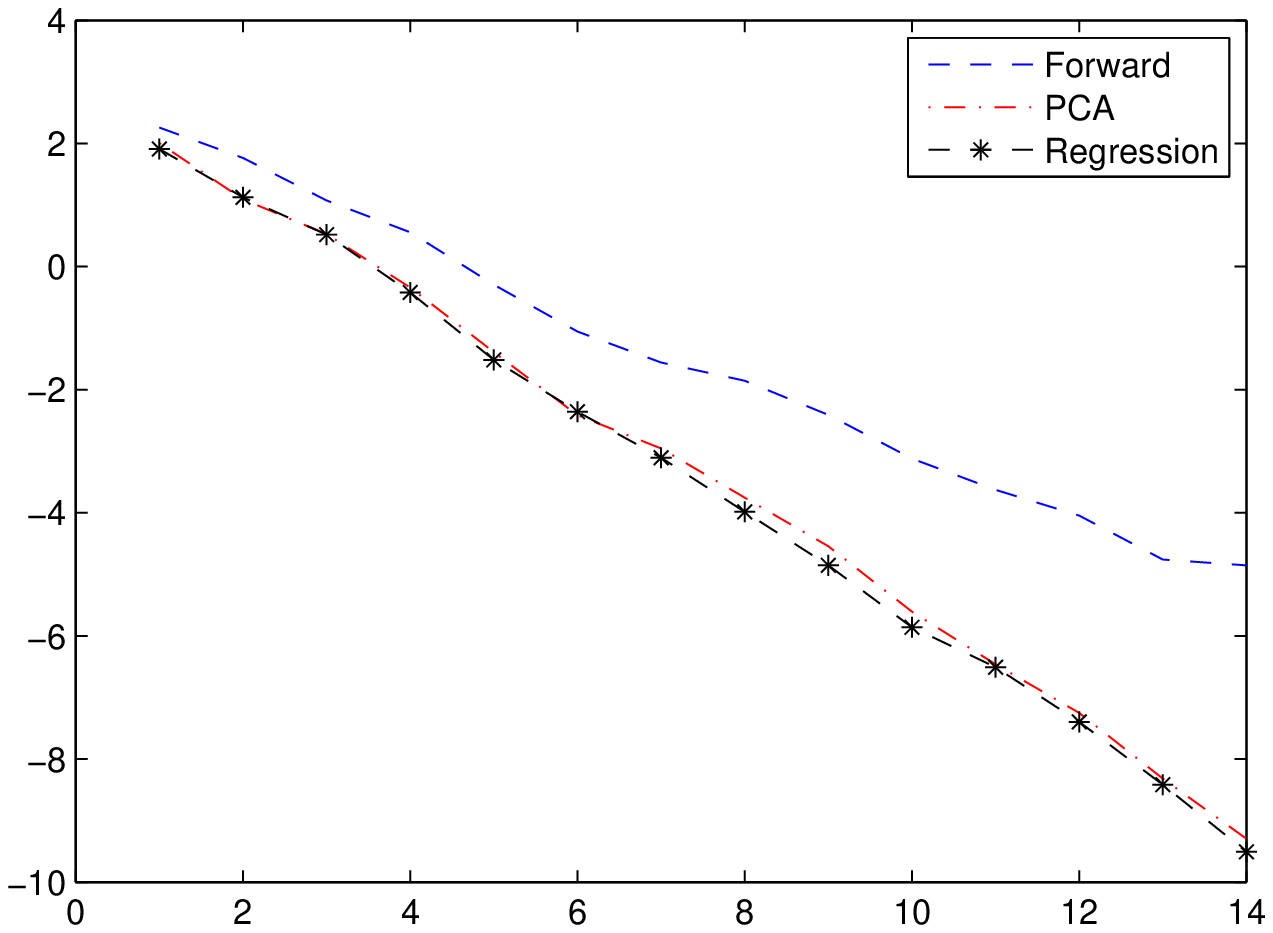}
\end{minipage}
\end{center}
\caption{Digital barrier option (left) and Asian barrier option (right): Standard deviation of $32$ runs on a $\log_2$-scale}
\label{fig:barrier}
\end{figure}

\subsection*{Asian barrier option}

The last example we provide is an Asian (up-and-in) barrier option by which we mean that the payoff of the option is similar to an Asian option as in the first numerical example, but is paid only if the underlying asset breaks through an upper barrier $u$. The corresponding function is then given by
\begin{align*}
f(X)=\exp(-rT)1_{\max_{k=1,\dots,n} S_{k}(X)\geq u}\max\left(\sum_{k=1}^{n}\frac{1}{n}S_{k}(X)-K,0\right)
\end{align*}
where $S_{k}(X)$ is as in (\ref{eq:discretestock}) for $k=1,\dots,n$. Since the function $f$ is of the form $f(X)=g(h_1(X),h_2(X))$ with $g(x,y)=\exp(-rT)xy$, $h_1(X)=1_{\max_{k=1,\dots,n} S_{k}(X)\geq u}$ and $h_2(X)=\max(\sum_{k=1}^{n}\frac{1}{n}S_{k}(X)-K,0)$, we apply Algorithm \ref{alg:general} with $m=2$ to the problem. The computation of the vectors $a^{(1)}$ and $a^{(2)}$ is already discussed in the examples above, i.e.\ $a^{(1)}$ is related to the digital barrier option and $a^{(2)}$
corresponds to the Asian option.

The numerical test is based on $32$ batches and we again compare the standard deviation of the forward method, the PCA construction and the regression method for various numbers of sample paths $N$, ranging from $2^1$ to $2^{14}$. Moreover, we use a Sobol sequence in dimension $n=1000$ with a random shift and the parameters are $S_0=100, K=100, u=110, r=0.04, \sigma=0.2$ and $T=1$. Figure \ref{fig:barrier} shows that the regression method yields slightly better results than the PCA and that the forward method is behind the other two approaches.

\section*{Appendix: Regression for the maximum}
\addcontentsline{toc}{section}{Appendix: Regression for the maximum}

We give the computations needed for examples of barrier type, that is we want to compute $\E(h(X)X_i)$ where $h$ is some function of the maximum of a discrete Brownian path with drift $\nu$, i.e.
\[
h(X)=\tilde{h}\left(\max_k \left(B_\frac{kT}{n}+\nu \frac{kT}{n}\right)\right)\,,
\]
and where $B_\frac{kT}{n}=\sqrt{\frac{T}{n}}\sum_{j=1}^k X_j$. We make the approximation
\begin{align}
\nonumber\E\left(\tilde{h}\left(\max_k \left(B_\frac{kT}{n}+\nu \frac{kT}{n}\right)\right)X_i\right)&\approx\E\left(\tilde{h}\left(\max_{0\le s\le T} (B_s+\nu s)\right)
\sqrt{n}\left(B_\frac{iT}{n}-B_\frac{(i-1)T}{n}\right)\right)\\
\label{eq:barrier}&=\E\left(\tilde{h}\left(\max_{0\le s\le T} B^\nu_s\right)
\sqrt{n}\left(B^\nu_\frac{iT}{n}-B^\nu_\frac{(i-1)T}{n}-\nu\frac{T}{n}\right)\right)\,,
\end{align}
where $B^\nu$ denotes Brownian motion with drift $\nu\in \R$, i.e.\ $B^\nu_t:=B_t+\nu t$, $t\ge 0$. Moreover, let $M^{\nu}_{t,T}:=\max_{t\leq s\leq T}B^{\nu}_{s}$ and $M^\nu_t:=M^\nu_{0,t}$. At first we compute $\E(1_{M^\nu_T\ge u}f(B^\nu_t))$ for given $u>0$ and measurable $f$ with $\E(|f(B^\nu_t)|)<\infty$. Then we show how the expectation for more general $\tilde h$ can be computed using the first result.

We start with a simple calculation for a Brownian motion $B$ with drift 0 and let $M_t:=M^0_t$. For $u\ge 0$ we get, using the reflection principle for Brownian motion,
\begin{align*}
\E(1_{M_t\ge u}f(B_t))
&=\E(1_{M_t\ge u}1_{B_t\ge u}f(B_t))+\E(1_{M_t\ge u}1_{B_t< u}f(B_t))\\
&=\E(1_{B_t\ge u}f(B_t))+\E(1_{B_t\ge u}f(2u -B_t))\,.
\end{align*}
Next we make a Girsanov-type change of measure such that under the new measure $Q$ the Brownian motion $B^\nu$ with drift becomes a standard Brownian motion. So with $\frac{dQ}{dP}=e^{-\nu B_t-\frac{\nu^2}{2}t}$, that is $\frac{dP}{dQ}=e^{\nu B^\nu_t-\frac{\nu^2}{2}t},$
\begin{align}
\nonumber\E\left(1_{M^\nu_t\ge u}f(B^\nu_t)\right)
&=\E_Q\bigl(1_{M^\nu_t\ge u}f(B^\nu_t)e^{\nu B^\nu_t-\frac{\nu^2}{2} t}\bigr)\\
\nonumber&=\E_Q\bigl(1_{B^\nu_t\ge u}f(B^\nu_t)e^{\nu B^\nu_t-\frac{\nu^2}{2}t}\bigr)+\E_Q\bigl(1_{B^\nu_t\ge u}f(2u-B^\nu_t)e^{\nu (2u -B^\nu_t)-\frac{\nu^2}{2}t}\bigr)\\
\nonumber&=\E\bigl(1_{B^\nu_t\ge u}f(B^\nu_t)\bigr)
+\E_Q\bigl(1_{-B^\nu_t\ge u}f(2u+B^\nu_t)e^{\nu (2u +B^\nu_t)-\frac{\nu^2}{2}t}\bigr)\\
\label{eq:refl1}&=\E\bigl(1_{B^\nu_t\ge u}f(B^\nu_t)\bigr)
+e^{2u\nu}\E\bigl(1_{B^\nu_t\le -u}f(2u+B^\nu_t)\bigr)\,.
\end{align}
The next step is to consider $\E(1_{M^\nu_T\ge u}f(B^\nu_t))$ for $t<T$. Let $\{{\cal F}_t\}_{0\le t\le T}$ denote the standard filtration of $B$.
\begin{align}
\nonumber\E\bigl(1_{M^\nu_T\ge u}f(B^\nu_t)\bigl)
&=\E\bigl(\E(1_{M^\nu_T\ge u}f(B^\nu_t)|{\cal F}_t)\bigr)
=\E\bigl(f(B^\nu_t)\E(1_{M^\nu_T\ge u}|{\cal F}_t)\bigr)\\
\nonumber&=\E\bigl(f(B^\nu_t)\E(1_{M^\nu_t\ge u}+1_{M^\nu_t< u}1_{M^\nu_{t,T}\ge u}|{\cal F}_t)\bigr)\\
\label{eq:refl2}&=\E\bigl(f(B^\nu_t)1_{M^\nu_t\ge u}\bigr)+\E\bigl(1_{M^\nu_t< u}f(B^\nu_t)\E(1_{M^\nu_{t,T}\ge u}|{\cal F}_t)\bigr)\,.
\end{align}
We have already computed the first term. For the second term we note that by the Markov property of Brownian motion,
\[
\E(1_{M^\nu_{t,T}\ge u}|{\cal F}_t)
=\E(1_{M^\nu_{t,T}\ge u}|B^\nu_t)
=\E(1_{\max_{t\le s\le T}(B^\nu_s-B^\nu_t)\ge (u-B^\nu_t)}|B^\nu_t)\,.
\]
We can use our earlier result (\ref{eq:refl1}) with $f(x)\equiv 1$ to obtain
\[
\E(1_{\max_{t\le s\le T}(B^\nu_s-B^\nu_t)\ge (u-B^\nu_t)}|B^\nu_t)
=\Phi\Bigl(\frac{B^\nu_t-u-\nu(T-t)}{\sqrt{T-t}}\Bigr)(1+e^{2u\nu})\,.
\]
Let us write $g(u,x):=\Phi\bigl(\frac{x-u-\nu(T-t)}{\sqrt{T-t}}\bigr)(1+e^{2u\nu})$. Then, using (\ref{eq:refl1}) and (\ref{eq:refl2}) we obtain
\begin{align}
\nonumber\E(1_{M^\nu_T\ge u}f(B^\nu_t))=&\E(f(B^\nu_t)1_{M^\nu_t\ge u})+\E(1_{M^\nu_t< u}f(B^\nu_t)g(u,B^\nu_t))\\
\nonumber=&\E(f(B^\nu_t)g(u,B^\nu_t))+\E(f(B^\nu_t)1_{M^\nu_t\ge u})-\E(1_{M^\nu_t\ge u}f(B^\nu_t)g(u,B^\nu_t))\\
\nonumber=&\E(f(B^\nu_t)g(u,B^\nu_t))+\E\left(1_{M^\nu_t\ge u}f(B^\nu_t)(1-g(u,B^\nu_t))\right)\\
\nonumber=&\E(f(B^\nu_t)g(u,B^\nu_t))+\E\left(1_{B^\nu_t\ge u}f(B^\nu_t)(1-g(u,B^\nu_t))\right)\\
&+e^{2u\nu}\E\left(1_{B^\nu_t\le -u}f(2u+B^\nu_t)(1-g(u,2u+B^\nu_t))\right)\,.\label{eq:refl3}
\end{align}
Note that the expectations can be computed explicitly for suitable $f$. 

We can also use \eqref{eq:refl3} to compute $\E(h(M^\nu_T)f(B^\nu_t))$ for $h$ differentiable and $h(0)=0$ and such that the expectations all converge absolutely:
\begin{align}
\nonumber\E(h(M^\nu_T)f(B^\nu_t))=&\E(\E(h(M^\nu_T)|B^\nu_t)f(B^\nu_t))\\
\nonumber=&\E\Bigl(\int_0^\infty h'(u)\E(1_{M^\nu_T\ge u}|B^\nu_t)d\!uf(B^\nu_t)\Bigr)\\
\nonumber=&\int_0^\infty h'(u)\E(\E(1_{M^\nu_T\ge u}|B^\nu_t)f(B^\nu_t))d\!u\\
\nonumber=&\int_0^\infty h'(u)\E(1_{M^\nu_T\ge u}f(B^\nu_t))d\!u\,.
\end{align}


\end{document}